\documentclass[12pt,a4paper]{article}
\usepackage{amsfonts,amssymb,amsmath}
\usepackage{theorem}
\usepackage[english]{babel} 

{\theorembodyfont{\rm}
  \newtheorem{defi}{Definition}[section]
  \newtheorem{exa}[defi]{Example}
}
  \newtheorem{thm}[defi]{Theorem}

\newcommand{\para}{\par\vspace{1ex plus0.2ex minus0.2ex}
\refstepcounter{defi}\textbf{\arabic{section}.\arabic{defi}}\, }

\newcommand{\cG}{{\mathcal G}}
\newcommand{\cL}{{\mathcal L}}
\newcommand{\End}{{\mathrm{End}}}
\newcommand{\dis}
                 {{\mathrel{\scriptstyle{\triangle}}}}
\newcommand{\notdis}{{\not\!\!\dis}}
\newcommand{\dist}{{\mathrm{dist}}}
\let\phi=\varphi

\let\theta=\vartheta

\newenvironment{proof}
    {\begin{trivlist} \item {\sl Proof:}} 
    {{}\hfill $\square$ \end{trivlist}} 

\begin{document}
\title{On bijections that preserve complementarity of subspaces}
\author{Andrea Blunck \and Hans Havlicek}
\date{}


\maketitle

\begin{abstract}
The set $\cG$ of all $m$-dimensional subspaces of a $2m$-dimensional
vector space $V$ is endowed with two relations, complementarity and
adjacency. We consider bijections from $\cG$ onto $\cG'$, where
$\cG'$ arises from a $2m'$-dimensional vector space $V'$. If such a
bijection $\phi$ and its inverse leave one of the relations from
above invariant, then also the other. In case $m\geq 2$ this yields
that $\phi$ is induced by a semilinear bijection from $V$ or from the
dual space of $V$ onto $V'$.
\par
As far as possible, we include also the infinite-dimensional case
into our considerations.
\end{abstract}

\section{Introduction}\label{se:introduction}

The purpose of the present article is to characterize the semilinear
bijections between vector spaces in terms of their action on certain
sets of subspaces. If $V$ is vector space (of finite or infinite
dimension) over a (not necessarily commutative) field $K$ then the
set of all subspaces $X$ of $V$ that are isomorphic to the quotient
space $V/X$ is denoted by $\cG$. We rule out all vector spaces of
finite odd dimension, since then $\cG$ is empty. It will be
convenient to turn $\cG$ into an undirected graph with vertex set
$\cG$; two vertices form an edge, whenever they are complements of
each other. This gives the \emph{distant graph} on $\cG$ (see
\ref{para:distant}). We adopt this name, as this graph is isomorphic
to the distant graph of the projective line over the ring
$\End_K(U)$, where $U$ is any element of $\cG$. Recall that two
points of a projective line over a ring are called \emph{distant} if,
and only if, they arise from a basis of the underlying module. We
shall also make use of the well known \emph{Grassmann graph} on
$\cG$; cf.\ \ref{para:adj}. It has the same set of vertices as the
distant graph, but $X,Y\in\cG$ comprise an edge if, and only if, they
are \emph{adjacent}, i.e., both $X$ and $Y$ have codimension $1$ in
$X+Y$.
 \par
Suppose now that
\begin{equation*}
 \phi:\cG\to\cG'
\end{equation*}
is an isomorphism of the distant graphs arising from vector spaces
$V$ and $V'$, respectively. So $\phi$ preserves complementarity in
both directions. It follows that $\phi$ is also an isomorphism of
Grassmann graphs; this will be shown in Theorem
\ref{thm:iso-vergleich} (a), which in turn is based upon a
characterization of adjacency in terms of the distant graph
(Theorem \ref{thm:hauptsatz}). Up to this point the dimension of
$V$ plays no role. But now we have to distinguish two cases:
\par 
If the dimension of $V$ is finite then the isomorphisms $\cG\to\cG'$ of the
Grassmann graphs are well known: Apart from two trivial cases ($\dim V=0,2$)
they arise from semilinear bijections $V\to V'$ or from semilinear bijections
$V^*\to V'$, where $V^*$ denotes the dual space of $V$. This result is due to
\textsc{W.-L.~Chow} \cite{chow-49}. Hence for $4\leq\dim V<\infty$ the given
isomorphism $\phi$ is induced by a semilinear bijection. Also, in the finite
dimensional case, every isomorphism $\cG\to\cG'$ of Grassmann graphs is an
isomorphism of distant graphs (Theorem \ref{thm:iso-vergleich} (b)). This leads
then to a complete description of the isomorphisms of distant graphs (Theorem
\ref{thm:comp-iso}).
\par
If the dimension of $V$ is infinite then so is the dimension of $V'$,
but at this moment it remains open whether or not the isomorphism
$\phi$ from above is induced by a semilinear bijection. There are two
reasons for this: Firstly, we do not know an algebraic description of
all isomorphisms of the corresponding Grassmann graphs, secondly, an
automorphism of the Grassmann graph on $\cG$ need not be an
automorphism of the distant graph on $\cG$ (Example \ref{exa:1}).

\section{Preliminaries}\label{se:preliminaries}
\para\label{para:distant}
Let $V$ be a left vector space over a (not necessarily
commutative) field $K$ and denote by $\cG$ the set of those
subspaces $X$ of $V$ that are isomorphic to the quotient space
$V/X$. Clearly, this condition is equivalent to saying that $X$ is
isomorphic to one (and hence all) of its complements with respect
to $V$. We assume that $\cG\neq\emptyset$, but there is no other
restriction on the dimension of $V$.

So, if $\dim V$ is finite, then it is an even number $2m$, say, and
the elements of $\cG$ are just the $m$-dimensional subspaces of $V$,
whence all elements of $\cG$ form an orbit under the action of the
general linear group of $V$.

If $\dim V$ is infinite then $\cG$ is non-empty and, as before, it is
an orbit under the action of the general linear group: For, if
$\mathfrak m$ denotes the cardinality of a basis of $V$ then
$\mathfrak m+\mathfrak m=\mathfrak m$ shows that $V\times V$ is
isomorphic to $V$. But in $V\times V$ the subspaces $V\times 0$ and
$0\times V$ are complementary and isomorphic, whence
$\cG\neq\emptyset$. Now let $X,Y\in\cG$. So there are subspaces
$X_1,Y_1\in\cG$ such that $V=X\oplus X_1=Y\oplus Y_1$. Then $X$,
$X_1$, $Y$, and $Y_1$ have bases with cardinality $\mathfrak m$,
whence they are mutually isomorphic. But the direct sum of any two
$K$-linear bijections $X\to Y$ and $X_1\to Y_1$ is a $K$-linear
bijection $V\to V$ taking $X$ to $Y$.

 \para
We say that $X,Y\in\cG$ are \emph{distant} (in symbols:
$X\,\dis\,Y$) whenever they are complementary, i.e., $X\oplus
Y=V$. The \emph{distant graph} on $\cG$ is the graph whose vertex
set is $\cG$ and whose edges are the unordered pairs of distant
elements. More generally, a distant graph can be associated with
every projective line over a ring \cite[p.~108]{blu+h-01a}. The
distant graph from above is -- up to isomorphism -- the distant
graph of the projective line over a ring $\End_K(U)$, where
$U\in\cG$. For a proof we refer to \cite[Theorem 2.1]{blunck-99}.
The particular case that this endomorphism ring is a
finite-dimensional $K$-algebra is treated in \cite[2.3]{thas-71},
where distant points are said to be ``in clear position'', and in
\cite[4.5. Example~(4)]{herz-95}. The distant graph is always
connected and its diameter is given in the following table:
\begin{equation}
\begin{array}{|c|c|c|c|c|}
  \hline
  \dim V &  \,0\, & \,2\, & 4,6,\ldots & \infty \\
  \hline
  \mbox{Diameter} & 0 & 1 & 2 &3\\
  \hline
\end{array}
\end{equation}
This is immediate for $\dim V<\infty$ and follows from \cite[Theorem
5.3]{blu+h-01a} when $\dim V=\infty$. The distant graph has no loops
except for $\dim V=0$.

\para\label{para:adj}
Two elements $X,Y\in\mathcal G$ are called \emph{adjacent} (in
symbols: $X\sim Y$) if
\begin{equation}\label{eq:adjacent1}
    \dim ((X+Y)/X) = \dim ((X+Y)/Y)=1.
\end{equation}
This terminology goes back to \textsc{W.-L.\ Chow} \cite{chow-49} in
the finite-dimensional case. Clearly, adjacency is an antireflexive
and symmetric relation. The \emph{Grassmann graph} on $\cG$ is the
graph whose vertex set is $\cG$ and whose edges are the $2$-sets of
adjacent vertices. (References are given in \ref{para:lit}.) So the
graph theoretic adjacency\footnote{In order to avoid ambiguity we
shall refrain from speaking of ``adjacent vertices'' of the distant
graph.} coincides with the relation $\sim$ defined according to
(\ref{eq:adjacent1}). By a simple induction, the following formula
for the distance of $X,Y\in\cG$ in the Grassmann graph can be
obtained:
\begin{equation}\label{eq:abstand}
\begin{array}{rcl}
  \dist(X,Y)=d &\Leftrightarrow& \dim ((X+Y)/X) =
                                 \dim ((X+Y)/Y) = d
  \\
               &\Leftrightarrow&\dim (X/(X\cap Y)) =
                                \dim (Y/(X\cap Y)) = d
\end{array}
\end{equation}
If $\dim V=2m$ is finite then
\begin{equation*}
 \dim ((X+Y)/X) = \dim ((X+Y)/Y) < \infty
\end{equation*}
is true for all $X,Y\in\cG$, whence the Grassmann graph is connected.
Obviously, its diameter is equal to $m$. We refer also to
\cite[4.4]{herz-95}.
 \par
If $\dim V$ is infinite then there are always subspaces $X,Y\in\cG$
such that
\begin{equation*}
 \dim ((X+Y)/X))\neq\dim ((X+Y)/Y),
\end{equation*}
even if both dimensions are finite. Let, for example, $Y$ be a
hyperplane of $X\in\cG$. So there exists a $1$-dimensional subspace
$A$ such that $X=Y\oplus A$ and there is an $X_1\in\cG$ with
$V=X\oplus X_1$. Hence $V=Y\oplus(A\oplus X_1)$. Since $\dim Y =
\dim(A\oplus X_1)$, we obtain that $Y\in\cG$. Now, obviously,
\begin{equation*}
  \dim ((X+Y)/X))=0\neq 1=\dim ((X+Y)/Y).
\end{equation*}
This explains why both the second and the third condition in
equation (\ref{eq:abstand}) involve two equations. Also, it
follows that the Grassmann graph is not connected: In fact, the
connected component of $X\in\cG$ is formed by all subspaces
$Y\in\cG$ satisfying (\ref{eq:abstand}) for some integer $d\geq
0$. Thus, for example, every complement of $X$ or every element
$Y\in\cG$ with\footnote{We use the sign $\leq$ for the inclusion
of subspaces and reserve $<$ for strict inclusion.} $Y<X$ or $Y>X$
is not in this connected component. Moreover, for each $d\geq 0$
there are a $d$-dimensional subspace $A\leq X$ and a subspace
$B\leq X$ such that $A\oplus B=X$. Also, there are a complement
$Y$ of $X$, a $d$-dimensional subspace $C\leq Y$, and a subspace
$D$ such that $Y=C\oplus D$. Then $C\oplus B\in\cG$, since
$A\oplus D$ is an isomorphic complement, and the distance of
$C\oplus B$ and $X$ is $d$. So in this case all connected
components of the Grassmann graph have infinite diameter.
\par
If we add a loop at each vertex of the Grassmann graph then we obtain
a \emph{Pl\"ucker space} in the sense of \textsc{W.~Benz}
\cite[p.~199]{benz-92}.
 \par

 \para\label{para:pls}
Let $M,N$ be subspaces of $V$ such that there is a $Y\in\cG$ with
\begin{equation*}
  M\leq Y\leq N \mbox{ and } \dim(Y/M)=\dim(N/Y)=1.
\end{equation*}
Then
\begin{equation}\label{eq:bueschel}
  \cG[M,N]:=\{X\in\cG\mid M< X< N\}
\end{equation}
is called a \emph{pencil} in $\cG$. If $\dim V=2m$ is finite then
$\dim M=m-1$ and $\dim N=m+1$, whence $M,N\notin\cG$. However, when
$V$ is infinite-dimensional, then it follows that $M,N\in\cG$. But
the strict inclusion signs in (\ref{eq:bueschel}) guarantee that
neither $M$ nor $N$ belongs to $\cG[M,N]$. If $\cL$ denotes the set
of all pencils in $\cG$ then $(\cG,\cL)$ is easily seen to be a
partial linear space with ``point set'' $\cG$ and ``line set'' $\cL$.
Two elements of $\cG$ are adjacent if, and only if, they are distinct
``collinear points'' of the partial linear space $(\cG,\cL)$. Every
``line'' is -- up to isomorphism -- a projective line over $K$,
whence it has $\#K +1$ ``points''. There are three essentially
different cases:
 \par
$0\leq\dim V=2m\leq 2$: Then $(\cG,\cL)$ is an $m$-dimensional
projective space over $K$.
 \par
$2<\dim V=2m<\infty$: Then $(\cG,\cL)$ is an example of a
\emph{Grassmann space}. It is a connected proper partial linear
space; see \cite{bichara+t-83}.
 \par
$\dim V=\infty$: It seems that in this case the partial linear space
$(\cG,\cL)$ has not been discussed in the literature so far. An
essential difference to the previous case is that $(\cG,\cL)$ is not
connected: Two ``points'' $X,Y\in\cG$ can be joined by a ``polygonal
path'' if, and only if, they are in the same connected component of
the Grassmann graph.

\para\label{para:lit}
There is a widespread literature dealing with characterizations of
(finite) Grassmann graphs and Grassmann spaces which are based upon
the set $\cG_{m,n}$ of all $m$-dimensional subspaces of an
$n$-dimensional vector space, $2\leq m\leq n-2<\infty$. We shall not
be concerned with such results, but we refer to \cite{bichara+t-83},
\cite[p.~268--272]{brouwer+cohen+neu-89}, \cite{ferrara+m-99a},
\cite{fu+huang-94}, \cite{metsch-95}, \cite{numata-90},
\cite{shult-94}, \cite{tallini-86}, \cite{vanBon+cohen-89} and, in
addition, to the many other papers which are cited there.
\par
The problem to determine and characterize all isomorphisms (or
automorphisms) of Grassmann graphs and Grassmann spaces has been
studied by many authors. See \cite[Kapitel 5]{benz-92},
\cite{brau-88}, \cite{chow-49}, \cite[p.~81]{dieu-71},
\cite{ferrara+m-99b} \cite{havl-95}, \cite{huang-98},
\cite{kreuzer-98}, \cite{pankov-02y}, \cite{pankov-02z}, and
\cite[p.~155]{wan-96}.

\section{A characterization of adjacency}\label{se:adjacency}

\para
In order to show the announced result on isomorphisms of distant
graphs we describe the adjacency relation in terms of the distant
graph. We shall use the terminology of projective geometry, i.e., the
one-, two-, and three-dimensional subspaces of $V$ will be called
\emph{points}, \emph{lines}, and \emph{planes}, respectively.

\begin{thm}\label{thm:hauptsatz}
  For all $P,Q\in\cG$ the following statements are equivalent:
\begin{enumerate}
  \item $P$ and $Q$ are adjacent.
  \item There is an element $R\in\cG$ satisfying the following
  conditions:
\begin{eqnarray}
\label{eq:hauptsatz1}
  R&\neq& P,Q,\\
\label{eq:hauptsatz2}
  \forall X\in \cG: X\,\dis\,R
  &\Rightarrow &
  X\,\dis\, P \mbox{ or } X\,\dis\,Q.
\end{eqnarray}
\end{enumerate}
\end{thm}

\begin{proof}
(a) $\Rightarrow$ (b): By $P\sim Q$, (\ref{eq:abstand}) holds for
$d=1$. So the set $\cG[P\cap Q,P+Q]$ is a pencil containing $P,Q$.
This pencil contains an element, say $R$, which is different from $P$
and $Q$, since every pencil contains $\# K+1$ elements.
 \par
Let $X\,\dis\,R$. Then $X\cap (P+Q)=:A$ is a point, since
$\dim((P+Q)/R)=1$. We deduce from $P\cap Q< R$ that the point $A$
cannot lie in both $P$ and $Q$. So, for example,
\begin{equation}\label{eq:XcapP}
  X\cap P=0,
\end{equation}
whence $P+Q=A\oplus P$. This gives
\begin{equation}\label{eq:X+P}
  V= X\oplus R \leq X+(P+Q) = X+(A\oplus P)=(X+A)+P=X+P.
\end{equation}
By (\ref{eq:XcapP}) and (\ref{eq:X+P}), we have $X\,\dis\,P$, as
required.
 \par
(b) $\Rightarrow$ (a): We proceed by showing several assertions. Some
of them are symmetric with respect to $P$ and $Q$, so it will be
sufficient to treat the assertion for $P$.
\par
(i) The first step is to establish that if $A,B\leq V$ then
\begin{equation}\label{eq:(i)}
  0\neq A\leq P\mbox{ and } 0\neq B\leq Q
  \Rightarrow
  (A+B)\cap R\neq 0.
\end{equation}
Assume to the contrary that $(A+B)\cap R=0$. We infer that there
exists a complement $X$ of $R$ containing $A+B$, whence
\begin{eqnarray*}
 0\neq A\leq X\cap P & \Rightarrow & X\,\notdis\,P, \\
 0\neq B\leq X\cap Q & \Rightarrow & X\,\notdis\,Q,
\end{eqnarray*}
which contradicts (\ref{eq:hauptsatz2}).

(ii) Next we claim that if $A,B\leq V$ then
\begin{equation}\label{eq:(ii)}
P\leq A\neq V\mbox{ and }Q\leq B\neq V \Rightarrow(A\cap B)+R\neq V.
\end{equation}
Assume to the contrary that $(A\cap B)+ R=V$. We infer that there
exists a complement $X$ of $R$ contained in $A\cap B$, whence
\begin{eqnarray*}
  X+P\leq A\neq V & \Rightarrow & X\,\notdis\,P, \\
  X+Q\leq B\neq V & \Rightarrow & X\,\notdis\,Q,
\end{eqnarray*}
which contradicts (\ref{eq:hauptsatz2}). Clearly, if $\dim V$ is
finite then (\ref{eq:(ii)}) follows from (\ref{eq:(i)}) by the
principle of duality.
\par
(iii) We show that
\begin{equation}\label{eq:(iii)}
 P\cap Q\leq R.
\end{equation}
This is true for $P\cap Q=0$. Otherwise choose any point $A$ in
$P\cap Q$. By (\ref{eq:(i)}), applied to $A=B$, we get $A\cap R\neq
0$. Hence $A\leq R$, and since $A\leq P\cap Q$ can be chosen
arbitrarily, $P\cap Q\leq R$.
 \par
(iv) We claim that
\begin{equation}\label{eq:(iv)}
 P+Q\geq R.
\end{equation}
This is true for $P+Q=V$. Otherwise choose any hyperplane $A$ through
$P+Q$. By (\ref{eq:(ii)}), applied to $A=B$, we get $A+R\neq V$.
Hence $A\geq R$, and $P+Q\geq R$ since $P+Q$ is equal to the
intersection of all such hyperplanes.
 \par
(v) Our next assertion is that
\begin{equation}\label{eq:(v)}
  P\not\leq Q \mbox{ and }Q\not\leq P.
\end{equation}
Assume to the contrary that $P\leq Q$ so that $P< R< Q$ follows from
(\ref{eq:(iii)}), (\ref{eq:(iv)}), and (\ref{eq:hauptsatz1}). So
there is a point $B\leq Q$ with $B\not\leq R$, and there exists a
complement $X$ of $R$ containing $B$. By the law of modularity,
\begin{equation*}
  (P+X)\cap R = P + (X\cap R) = P+0 = P\neq R,
\end{equation*}
whence $P+X\neq V$. Consequently, $X\,\notdis\,P$. Furthermore,
$0\neq B\leq X\cap Q$ shows that $X\,\notdis\,Q$. Altogether, this
contradicts (\ref{eq:hauptsatz2}).
\par
(vi) We continue by showing that
\begin{equation}\label{eq:(vi)}
 P\not\leq R\mbox{ and }Q\not\leq R.
\end{equation}        
Assume to the contrary that $P\leq R$, whence (\ref{eq:hauptsatz1}) yields $P<
R$. We know from (\ref{eq:(v)}) that $P\not\leq Q$, whence there is a
hyperplane $H$ containing $Q$ with $P\not\leq H$. Thus $V=P+H\leq R+H\leq V$.
So there is a complement $X$ of $R$ which lies in $H$. We have $X+Q\leq H$ and,
consequently, $X\,\notdis\,Q$. As in the proof of equation (\ref{eq:(v)}), the
strict inclusion $P< R$ implies that no complement of $R$ can be a complement
of $P$. This gives $X\,\notdis\,P$ which is absurd by (\ref{eq:hauptsatz2}).
 \par
(vii) Now it is our task to verify that
\begin{equation}\label{eq:(vii)}
 \dim (P/(P\cap R))=\dim (Q/(Q\cap R))=1.
\end{equation}
By (\ref{eq:(vi)}), $P\not\leq R$ so that $P\cap R\neq P$ and $\dim
P\geq 1$. So, for $\dim P=1$, we obtain $\dim (P/(P\cap R))=1$. If
$\dim P\geq 2$ then it suffices to show that the inequality
\begin{equation*}
 L\cap(P\cap R)\neq 0
\end{equation*}
holds for all lines $L\leq P$. We deduce from (\ref{eq:(vi)}) that
there is a point $B\leq Q$ with $B\not\leq R$. By (\ref{eq:(iii)}),
$B\not\leq P$ and so for every line $L\leq P$ the subspace $L\oplus
B$ is a plane. Let $A_1,A_2\leq L\leq P$ be distinct points. Each
point $A_i$ together with the point $B\leq Q$ meets the requirements
of (\ref{eq:(i)}). This shows that
\begin{equation*}
  C_i:=(A_i\oplus B)\cap R
\end{equation*}
is a point other than $B$ for $i=1,2$. As $A_1$ and $A_2$ are
different, so are $C_1$ and $C_2$. Hence the subspace $C_1\oplus
C_2\leq R$ is a line. But $L$ and $C_1\oplus C_2$ are coplanar,
whence they have a common point which lies in $L\cap (P\cap R)$. (See
the figure below which illustrates the general case when $A_1\neq
C_1$ and $A_2\neq C_2$.)
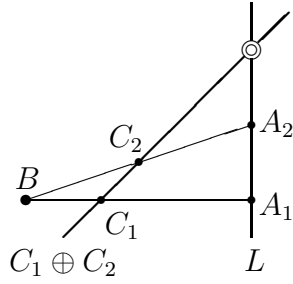
\begin{center}
  \unitlength5mm
\begin{picture}(7,8)(0,-2)
\thinlines
  \put(0,0){\line(1,0){6}}
  \put(0,0){\line(3,1){6}}
\thicklines
  \put(6,-1){\line(0,1){4.75}}
  \put(6,4.25){\line(0,1){1.0}}

  \put(1,-1){\line(1,1){4.81}}
  \put(6.19,4.19){\line(1,1){0.8}} 
\thinlines

  \put(0,0){\circle*{0.3}} 

  \put(6,4){\circle{0.25}}
  \put(6,4){\circle{0.5}}

  \put(6,0){\circle*{0.2}}
  \put(2,0){\circle*{0.2}}
  \put(6,2){\circle*{0.2}}
  \put(3,1){\circle*{0.2}}
  \put(-0.3, 0.3){$B$}
  \put(2.1,-0.9){$C_1$}
  \put(2.2, 1.4){$C_2$}
  \put(6.2,-0.4){$A_1$}
  \put(6.2, 1.8){$A_2$}
  \put(-0.45,-1.9){$C_1\oplus C_2$}
  \put(5.8,-1.9){$L$}
\end{picture}
\end{center}

\par
(viii) The penultimate step is to show that
\begin{equation}\label{eq:(viii)}
  \dim ((P+Q)/R)=1.
\end{equation}
By (\ref{eq:(vii)}) there are points $A,B$ with
\begin{eqnarray*}
 P&=&(P\cap R)\oplus A,\\
 Q&=&(Q\cap R)\oplus B.
\end{eqnarray*}
We cannot have $A=B$, since then (\ref{eq:(iii)}) would give $A=B\leq
P\cap Q\leq R$ which is impossible because obviously
\begin{equation*}
  A,B\not\leq R.
\end{equation*}
So (\ref{eq:(i)}) yields that $C:= (A\oplus B)\cap R$ is a point. As
the point $C$ lies in $R$, it is different from $A$ and $B$. Next, it
follows that
\begin{equation*}
\begin{array}{rcll}
  P+Q & =         & (P\cap R) + (Q\cap R)+A+B &\\
      & =         & (P\cap R) + (Q\cap R)+C+B &\mbox{(by the exchange lemma)}\\
      & \leq & R +C+B &\mbox{(by $P\cap R\leq R$ and $Q\cap R\leq R$)}\\
      & = & R\oplus B &\mbox{(by $C\leq R$ and $B\not\leq R$)}\\
      & \leq & P+Q &\mbox{(by (\ref{eq:(iv)}) and $B\leq Q$)}.
\end{array}
\end{equation*}
Therefore $P+Q=R\oplus B$.
\par
(ix) Now, finally, we are in a position to show that $P$ and $Q$ are
adjacent. This is equivalent, by definition, to
\begin{equation}\label{eq:(ix)}
  \dim((P+Q)/P)=\dim((P+Q)/Q)=1.
\end{equation}     
We infer from the proof of (\ref{eq:(viii)}) that there is a point $B\leq Q$
such that $P+Q=R\oplus B$. Denote by $D$ a complement of $P+Q$. (The following
is trivial if $P+Q=V$, since then $D=0$.) So we get
\begin{equation*}
   V=(P+Q)\oplus D = (R\oplus B)\oplus D,
\end{equation*}
whence $X:=D\oplus B$ is a complement of $R$. But $B\leq X\cap Q$
yields $X\,\notdis\,Q$. Now (\ref{eq:hauptsatz2}) forces that
$X\,\dis\,P$. So, the three formulas
\begin{eqnarray*}
  V=X\oplus P &=& D\oplus (P\oplus B),\\
  V&=& D\oplus (P+Q),\\
  P\oplus B &\leq& P+Q,
\end{eqnarray*}
together yield that $P\oplus B = P+Q$, whence $\dim((P+Q)/P)=1$, as
required.
 \par
This completes the proof.
\end{proof}

Let us remark that the adjacency relation can be expressed in terms
of the distant graph in a trivial way in the following cases: For
$\dim V=0$ we have $X\sim Y \Leftrightarrow X\,\notdis\,Y$, since
both sides are identically false. For $\dim V=2$ we have $X\sim Y
\Leftrightarrow X\,\dis\,Y$. For $\dim V=4$ we have $X\sim Y
\Leftrightarrow X\,\notdis\,Y\neq X$.

\section{Isomorphisms of distant graphs}\label{se:isomorphisms}

 \para\label{para:semilin}
Suppose now that $V$ and $V'$ are left vector spaces over $K$ and
$K'$, respectively. As before, we assume that neither $\cG$ nor
$\cG'$ is empty. Let $f:V\to V'$ be a semilinear bijection. Such an
$f$ induces a bijection $\mathcal G\to \mathcal G'$ by $X\mapsto
X^f$.
 \par
Let $V^*$ denote the dual vector space of $V$; we consider $V^*$ as a
right vector space over $K$ and we assume now that $\dim V<\infty$.
If $f:V^*\to V'$ is a semilinear bijection (with respect to an
antiisomorphism $K\to K'$) then $X\mapsto (X^\perp)^f$ is a bijection
$\cG\to\cG'$, since the annihilator map $\perp$ yields a bijection of
$\cG$ onto $\cG^*$.
 \par
In both cases this bijection $\cG\to\cG'$ is an isomorphism of
distant graphs and an isomorphism of Grassmann graphs. Furthermore,
it is a collineation of the partial linear space $(\cG,\cL)$ onto
$(\cG',\cL')$.

\begin{thm}\label{thm:iso-vergleich}
Suppose that $V$ and $V'$ are left vector spaces over $K$ and $K'$,
respectively. Then the following assertions hold:
\begin{enumerate}
  \item
If $\phi:\cG\to\cG'$ is an isomorphism of distant graphs then it is
also an isomorphism of Grassmann graphs.
  \item
Suppose, moreover, that $\dim V=2m$ is finite. If $\phi:\cG\to\cG'$
is an isomorphism of Grassmann graphs then it is also an isomorphism
of distant graphs.
\end{enumerate}
\end{thm}

\begin{proof}
(a) This is an immediate consequence of the characterization of
adjacency in terms of the distant graph given in Theorem
\ref{thm:hauptsatz}.
\par
(b) Let $\phi:\cG\to\cG'$ be an isomorphism of Grassmann graphs. By
$\dim V=2m$ and \ref{para:adj}, the Grassmann graph on $\cG$ is
connected and its diameter is $m$. By virtue of $\phi$, the Grassmann
graph on $\cG'$ is also connected and it has the same diameter $m$.
We read off from \ref{para:adj} that $\dim V'=2m$. For all
$X,Y\in\cG$ we have $X\,\dis\,Y$ if, and only if, $X$ and $Y$ are at
distance $m$ in the Grassmann graph on $\cG$. Hence $X^\phi$ and
$Y^\phi$, too, are at distance $m$ in the Grassmann graph on $\cG'$
which in turn is equivalent to $X^\phi\,\dis\,Y^\phi$.
\end{proof}

We refer to \cite{pamb-00} for the logical background of our
reasoning in part (a) of the previous proof. For $\dim V=\infty$ the
assertion in (b) need not be true:

\begin{exa}\label{exa:1}
Let $\dim V=\infty$ and suppose that $P\,\dis\,Q$. Choose points
$A\leq P$ and $B\leq Q$. There are subspaces $A_1$ and $B_1$ such
that
\begin{equation*}
  P=A\oplus A_1,\; Q=B\oplus B_1.
\end{equation*}
Then $R:=B\oplus A_1\in\cG$, since it is isomorphic and complementary to
$A\oplus B_1$. There exists a $K$-linear bijection $f:V\to V$ taking $P$ to
$R$. We define the following map:
\begin{equation*}
  \phi:\cG\to\cG :
  \left\{
  \begin{array}{ll}
    X\mapsto X^f & \mbox{if }\dim (P/(X\cap P))=\dim (X/(X\cap P))<\infty \\
    X\mapsto X & \mbox{otherwise} \
  \end{array}
  \right.
\end{equation*}
This means that $f$ is applied to all elements of the connected
component of $P$ in the Grassmann graph, whereas all other elements
of $\cG$ remain fixed. As $f$ preserves adjacency and non-adjacency,
the connected component of $P$, which coincides with the connected
component of $R$, is mapped bijectively onto itself. So $\phi$ is an
automorphism of the Grassmann graph. However, $\phi$ is not an
automorphism of the distant graph, since
$P^\phi=R\,\notdis\,Q=Q^\phi$.
\end{exa}

\begin{thm}\label{thm:comp-iso}
Let $V$ and $V'$ be left vector spaces over $K$ and $K'$,
respectively, where $\dim V=2m$ is finite. A bijection
$\phi:\cG\to\cG'$ is an isomorphism of distant graphs if, and only
if, one of the following assertions holds:
\begin{enumerate}
  \item
$\dim V=\dim V'=0$.
  \item
$\dim V=\dim V'=2$ and $\# K=\# K'$.
  \item
$4\leq \dim V=\dim V'=2m<\infty$ and there is either a semilinear
bijection $f:V\to V'$ such that $X^\phi=X^f$ or a semilinear
bijection $f:V^*\to V'$ such that $X^\phi=(X^\perp)^f$.
\end{enumerate}
\end{thm}

\begin{proof}
Let $\phi$ be an isomorphism of distant graphs. By Theorem
\ref{thm:iso-vergleich}, $\phi$ is an isomorphism of Grassmann
graphs. Furthermore, we see from the proof of Theorem
\ref{thm:iso-vergleich} (b) that $\dim V=\dim V'=2m$.
 \par
If $\dim V=2$ then the distant graph on $\cG$ is a complete graph
with $\# K +1$ vertices. Hence the same properties are shared by the
isomorphic distant graph on $\cG'$. Therefore $\# K=\# K'$.
\par
If $\dim V\geq 4$ then the assertion follows from a theorem due to
\textsc{W.L.~Chow} on the isomorphisms of Grassmann graphs. See
\cite{chow-49} or \cite[p.~81]{dieu-71}.
\par
The converse is trivially true, if one of the assertions (a) or (b)
is satisfied. If (c) holds then $\phi$ is an isomorphism of distant
graphs according to \ref{para:semilin}.
\end{proof}

So only the case $\dim V=\dim V'=\infty$ remains open. In view of
Theorem \ref{thm:iso-vergleich} (a) and Example \ref{exa:1}, a
promising strategy could be as follows: First, describe all
isomorphisms of Grassmann graphs and then single out the isomorphisms
of distant graphs.
 \par
Another problem is as follows: Suppose that $\dim V\geq 4$ and that
$\phi:\cG\to\cG'$ is a bijection such that
\begin{equation*}
  X\,\dis\,Y \Rightarrow  X^\phi\,\dis\,Y^\phi
\end{equation*}
for all $X,Y\in\cG$. Is such a $\phi$ an isomorphism of distant
graphs? By \cite[Theorem 5.1]{blu+h-00c}, the answer is affirmative
if $\dim V=\dim V'=4$.


Andrea Blunck\\ Fachbereich Mathematik\\ Universit\"at Hamburg\\
Bundesstra{\ss}e 55\\ D--20146 Hamburg\\ Germany\\ email:
\texttt{andrea.blunck@math.uni-hamburg.de}
 \\[1em]
Hans Havlicek\\ Institut f\"ur Geometrie\\ Technische Universit\"at\\
Wiedner Hauptstra{\ss}e 8--10\\ A--1040 Wien\\ Austria\\ email:
\texttt{havlicek@geometrie.tuwien.ac.at}

\end{document}